\documentclass{amsproc}
\usepackage{breqn,float,amssymb,graphicx,caption}
\usepackage{array, makecell}   
\setcellgapes{5pt}             
\usepackage[skip=1ex]{caption} 
\makeatletter
\@namedef{subjclassname@2020}{%
  \textup{2020} Mathematics Subject Classification}
\makeatother
\newtheorem{theorem}{Theorem}[section]

\theoremstyle{definition}

\newtheorem{example}[theorem]{Example}

\theoremstyle{remark}

\numberwithin{equation}{section}
\allowdisplaybreaks
\begin{document}
\renewcommand{\arraystretch}{3}
\title[Extended Blumberg-Dieckmann Series]{Extended Blumberg-Dieckmann Series}


\author{Robert Reynolds}
\address[Robert Reynolds]{Department of Mathematics and Statistics, York University, Toronto, ON, Canada, M3J1P3}
\email[Corresponding author]{milver73@gmail.com}
\thanks{}




\date{}

\dedicatory{}

\begin{abstract}
This paper introduces a set of finite summation formulas and utilize them to establish various functional relationships involving the multivariable Hurwitz-Lerch zeta function. Additionally, the paper examines several examples of these functional relationships.
\end{abstract}

\maketitle
\section{Introduction}
The finite sum of Hurwitz-Lerch zeta functions in the context of an infinite array of line charges has been studied in the field of electrostatics and mathematical physics \cite{blumberg,dieckmann}. This specific topic falls under the broader area of potential theory and the calculation of electrostatic potentials generated by distributions of charges. The study of infinite arrays of line charges involves analyzing the behavior of electric potentials and fields resulting from an infinite sequence of evenly spaced line charges \cite{erik}. The Hurwitz-Lerch zeta function, which is a generalization of the Riemann zeta function, appears in the mathematical representation of these potentials. Researchers have investigated various properties and applications of the finite sum of Hurwitz-Lerch zeta functions in this context. They have explored convergence properties, section (25.14) in \cite{dlmf}, evaluated the sums for specific parameters, derived analytical expressions or approximations, section (25.20) in \cite{dlmf}, and examined the behavior of the potentials as the number of charges or their separation approaches infinity. These studies are important for understanding the behavior of electrostatic fields in periodic systems, such as crystal lattices or periodic structures in condensed matter physics \cite{betermin}. They provide insights into the potential distribution and energy calculations in these systems, aiding in the analysis and design of devices and materials. Specific research papers and academic publications on the finite sum of Lerch functions in the context of infinite arrays of line charges can be found by searching scholarly databases or consulting textbooks and research papers in the field of mathematical physics and electrostatics.\\\\

This paper introduces a table (\ref{tab:SumProdTables}) of formulas similar to a table of integral formulas \cite{bromwich}. The selection of formulas for this table was influenced by their relevance to number theory \cite{apostol}. A finite sum formula represents the sum of a finite series, and the evaluation of such a sum is known as finite summation. To achieve this, general sum formulas are developed and applied, providing closed-form expressions that yield the exact or approximate sum depending on the series type. The paper begins with the presentation of finite summation theory, followed by the development of both general and special finite sum formulas. Finally, the table itself is presented. In the context of studying an infinite array of line charges \cite{dieckmann}, functions involving products of trigonometric functions spaced at regular intervals arise, which are functions of the complex variable $z=x+iy$. It has been demonstrated that these functions can be simplified to a constant multiplied by a multiple angle function. In this work, we apply the contour integral method from \cite{reyn4}, to the finite sum of the secant function given in \cite{hansen, dieckmann}, resulting in
\begin{multline}\label{eq:contour}
\frac{1}{2\pi i}\int_{C}\sum_{j=0}^{2n}(-1)^j a^w w^{-k-1} \sec \left(\frac{\pi  j}{2 n+1}+m+w\right)dw\\
=\frac{1}{2\pi i}\int_{C}(-1)^n (2 n+1) a^w w^{-k-1} \sec ((2 n+1)
   (m+w))dw
\end{multline}
where $a,m,k\in\mathbb{C},Re(m+w)>0,n\in\mathbb{Z^{+}}$. Using equation (\ref{eq:contour}) the main Theorem to be derived and evaluated is given by
\begin{multline}
\sum_{j=0}^{2n}(-1)^j e^{i \left(\frac{\pi  j}{2 n+1}+m\right)} \Phi \left(-e^{2 i
   \left(\frac{\pi  j}{2 n+1}+m\right)},-k,\frac{1}{2} (1-i \log
   (a))\right)\\
   =i^{-k} (-1)^n (2 n+1) (i (2 n+1))^k e^{i m (2 n+1)} \Phi
   \left(-e^{2 i m (2 n+1)},-k,\frac{2 n-i \log (a)+1}{4 n+2}\right)
\end{multline}
where the variables $k,a,m$ are general complex numbers and $n$ is any positive integer. This new expression is then used to derive special cases in terms of trigonometric functions. The derivations follow the method used by us in \cite{reyn4}. This method involves using a form of the generalized Cauchy's integral formula given by
\begin{equation}\label{intro:cauchy}
\frac{y^k}{\Gamma(k+1)}=\frac{1}{2\pi i}\int_{C}\frac{e^{wy}}{w^{k+1}}dw,
\end{equation}
where $y,w\in\mathbb{C}$ and $C$ is in general an open contour in the complex plane where the bilinear concomitant \cite{reyn4} is equal to zero at the end points of the contour. This method involves using a form of equation (\ref{intro:cauchy}) then multiplies both sides by a function, then takes the finite sum of both sides. This yields a finite sum in terms of a contour integral. Then we multiply both sides of equation (\ref{intro:cauchy})  by another function and take the infinite sum of both sides such that the contour integral of both equations are the same.
\section{The Hurwitz-Lerch zeta Function}

We use equation (1.11.3) in \cite{erd} where $\Phi(z,s,v)$ is the Hurwitz-Lerch zeta function which is a generalization of the Hurwitz zeta $\zeta(s,v)$ and Polylogarithm function $\text{Li}_n(z)$. The Lerch function has a series representation given by

\begin{equation}\label{knuth:lerch}
\Phi(z,s,v)=\sum_{n=0}^{\infty}(v+n)^{-s}z^{n}
\end{equation}
where $|z|<1, v \neq 0,-1,-2,-3,..,$ and is continued analytically by its integral representation given by
\begin{equation}\label{knuth:lerch1}
\Phi(z,s,v)=\frac{1}{\Gamma(s)}\int_{0}^{\infty}\frac{t^{s-1}e^{-vt}}{1-ze^{-t}}dt=\frac{1}{\Gamma(s)}\int_{0}^{\infty}\frac{t^{s-1}e^{-(v-1)t}}{e^{t}-z}dt
\end{equation}
where $Re(v)>0$, and either $|z| \leq 1, z \neq 1, Re(s)>0$, or $z=1, Re(s)>1$.
\section{Contour Integral Representation for the Finite Sum of the Hurwitz-Lerch zeta Functions}
In this section we derive the contour integral representations of the left-hand side and right-hand side of equation (\ref{eq:contour}) in terms of the Hurwtiz-Lerch zeta and trigonometric functions.
\subsection{Derivation of the left-hand side contour integral}
We use the method in \cite{reyn4}. Using equation (\ref{intro:cauchy})  we first replace $\log (a)+i (2 y+1)$ and multiply both sides by $2(-1)^y e^{i (2 y+1) \left(\frac{\pi  j}{2 n+1}+m\right)}$ then take the finite and infinite sums over $j\in[0,2n]$ and $y\in [0,\infty)$ and simplify in terms of the Hurwitz-Lerch zeta function to get
\begin{multline}\label{fsci1}
\sum_{j=0}^{2n}\frac{(-1)^j i^k 2^{k+1} e^{i \left(\frac{\pi  j}{2 n+1}+m\right)} }{k!}\Phi \left(-e^{2 i \left(\frac{\pi  j}{2
   n+1}+m\right)},-k,\frac{1}{2} (1-i \log (a))\right)\\\
   =\frac{1}{2\pi i}\sum_{y=0}^{\infty}\sum_{j=0}^{2n}\int_{C}(-1)^y a^w e^{i b (2 y+1) (m+w)}dw\\\\
   =\frac{1}{2\pi i}\int_{C}\sum_{j=0}^{2n}\sum_{y=0}^{\infty}(-1)^y a^w e^{i b (2 y+1) (m+w)}dw\\\\
   =\frac{1}{2\pi i}\int_{C}\sum_{j=0}^{2n}(-1)^j a^w w^{-k-1} \sec \left(\frac{\pi  j}{2
   n+1}+m+w\right)dw
\end{multline}
from equation (1.232.2) in \cite{grad} where $Re(w+m)>0$ and $Im\left(m+w\right)>0$ in order for the sums to converge. We apply Tonelli's theorem for multiple sums, see page 177 in \cite{gelca} as the summands are of bounded measure over the space $\mathbb{C} \times [0,2n] \times [0,\infty)$.
\subsection{Derivation of the right-hand side contour integral}
We use the method in \cite{reyn4}. Using equation (\ref{intro:cauchy})  we first replace $\log (a)+i (2 n+1) (2 y+1)$ and multiply both sides by $2 (-1)^y e^{i m (2 n+1) (2 y+1)}$ then take the finite and infinite sums over  $y\in [0,\infty)$ and simplify in terms of the Hurwitz-Lerch zeta function to get
\begin{multline}\label{fsci2}
\frac{2^{k+1} (-1)^n (2 n+1) (i (2 n+1))^k e^{i m (2 n+1)}}{k!} \Phi \left(-e^{2 i m (2 n+1)},-k,\frac{2 n-i \log
   (a)+1}{2 (2 n+1)}\right)\\\\
   =\frac{1}{2\pi i}\sum_{y=0}^{\infty}\int_{C}2 (-1)^y a^w e^{i (2 n+1) (2 y+1) (m+w)}dw\\\\
   =\frac{1}{2\pi i}\int_{C}\sum_{y=0}^{\infty}2 (-1)^y a^w e^{i (2 n+1) (2 y+1) (m+w)}dw\\\\
   =\frac{1}{2\pi i}\int_{C}(-1)^n (2 n+1) a^w w^{-k-1} \sec ((2 n+1) (m+w))dw
\end{multline}
from equation (1.232.2) in \cite{grad} where $Re(w+m)>0$ and $Im\left(m+w\right)>0$ in order for the sums to converge. We apply Tonelli's theorem for multiple sums, see page 177 in \cite{gelca} as the summands are of bounded measure over the space $\mathbb{C} \times [0,\infty) $.
\section{A functional equation for the Hurwitz-Lerch zeta function and its evaluation}
In this section we will derive and evaluate formulae involving the sum and product of the Hurwitz-Lerch zeta function in terms other Special functions and fundamaental constants.
\begin{theorem}
For all $k,a,m \in\mathbb{C}$ then, 
\begin{multline}\label{fslf}
\sum_{j=0}^{2n}(-1)^j e^{i \left(\frac{\pi  j}{2 n+1}+m\right)} \Phi \left(-e^{2 i
   \left(\frac{\pi  j}{2 n+1}+m\right)},-k,\frac{1}{2} (1-i \log
   (a))\right)\\
   =i^{-k} (-1)^n (2 n+1) (i (2 n+1))^k e^{i m (2 n+1)} \Phi
   \left(-e^{2 i m (2 n+1)},-k,\frac{2 n-i \log (a)+1}{4 n+2}\right)
\end{multline}
\end{theorem}
\begin{proof}
With respect to equation (\ref{eq:contour}) and observing the addition of the right-hand sides of relation (\ref{fsci1}), and relation (\ref{fsci2}) are identical; hence, the left-hand sides of the same are identical too. Simplifying with the Gamma function yields the desired conclusion.
\end{proof}
\begin{example}
An alternate form: A functional equation for the Hurwitz-Lerch zeta function.
\begin{multline}
\sum_{j=0}^{2n}(-1)^{j \left(\frac{1}{2 n+1}+1\right)} \Phi \left(-e^{i \left(\frac{2
   \pi  j}{2 n+1}+m\right)},-k,a\right)\\
   =(-1)^n (2 n+1)^{k+1} e^{i m n} \Phi
   \left(-e^{i m (2 n+1)},-k,\frac{a+n}{2 n+1}\right)
\end{multline}
\end{example}
\begin{proof}
Use equation (\ref{fslf}) and rewrite using $a=e^{ai}$ and simplify.
\end{proof}
\begin{example}
The Degenerate Case.
\begin{equation}
\sum_{j=0}^{2n}(-1)^j \sec \left(\frac{\pi  j}{2 n+1}+m\right)=(-1)^n (2 n+1) \sec (2 m
   n+m)
\end{equation}
\end{example}
\begin{proof}
Use equation (\ref{fslf}) and set $k=0$ and simplify using entry (2) in Table below (64:12:7) in \cite{atlas}.
\end{proof}
\begin{example}
Hurwitz-Lerch zeta functional equation.
\begin{multline}\label{lfe}
\Phi (z,s,a)\\
=3^{1-s} z \Phi \left(z^3,s,\frac{a+1}{3}\right)-(-1)^{2/3} \Phi \left(-\sqrt[3]{-1}
   z,s,a\right)+\sqrt[3]{-1} \Phi \left((-1)^{2/3} z,s,a\right)
\end{multline}
\end{example}
\begin{proof}
Use equation (\ref{fslf}) and set $n= 1,m= -\frac{1}{2} i \log (-z),a= e^{i (2 a - 1)},k= -s$ and simplify.
\end{proof}
\begin{example}
Hurwitz-Lerch zeta functional equation.
\begin{multline}
\Phi (-z,s,a)\\
=5^{1-s} z^2 \Phi
   \left(-z^5,s,\frac{a+2}{5}\right)+(-1)^{3/5} \Phi \left(\sqrt[5]{-1}
   z,s,a\right)+\sqrt[5]{-1} \Phi \left(-(-1)^{2/5} z,s,a\right)\\
   -(-1)^{2/5}
   \left((-1)^{2/5} \Phi \left((-1)^{3/5} z,s,a\right)+\Phi \left(-(-1)^{4/5}
   z,s,a\right)\right)
\end{multline}
\end{example}
\begin{proof}
Use equation (\ref{fslf}) and set $n= 2,m= \frac{1}{2} i \log (z),a= e^{i (2 a - 1)},k= -s$ and simplify.
\end{proof}
\begin{example}
Special case of the first partial derivative of the Hurwitz-Lerch zeta function in terms of the log-gamma function.
\begin{multline}
\sqrt[3]{-1} \Phi'\left((-1)^{2/3},0,a\right)-(-1)^{2/3}
   \Phi'\left(-\sqrt[3]{-1},0,a\right)
   =\log \left(\frac{2 \pi  3^{\frac{1}{2}-a}
   \Gamma (a)}{\Gamma \left(\frac{a+1}{3}\right)^3}\right)
\end{multline}
\end{example}
\begin{proof}
Use equation (\ref{lfe}) set $z=1$ and take the first partial derivative with respect to $s$ and set $s=0$ and simplify using equation (25.11.18) in \cite{dlmf}.
\end{proof}
\begin{example}
Special case of the first partial derivative of the Hurwitz-Lerch zeta function in terms of the log-gamma function.
\begin{multline}
\sqrt[3]{-1}
   \Phi'\left(-(-1)^{2/3},0,a\right)-(-1)^{2/3}
   \Phi'\left(\sqrt[3]{-1},0,a\right)\\
   =3 \log
   \left(\frac{\sqrt{3} \sqrt[3]{\frac{\Gamma
   \left(\frac{a-2}{2}\right)}{\Gamma \left(\frac{a-1}{2}\right)}} \Gamma
   \left(\frac{a+1}{6}\right)}{\left((a-2) \sqrt{a-1}\right)^{2/3} \Gamma
   \left(\frac{a-2}{6}\right)}\right)+\log (2)
\end{multline}
\end{example}
\begin{proof}
Use equation (\ref{lfe}) set $z=-1$ and take the first partial derivative with respect to $s$ and set $s=0$ and simplify using equation (25.11.18) in \cite{dlmf}.
\end{proof}
\begin{example}
Finite product of the cosine and sine functions in terms of the ratio of the sine and tangent functions.
\begin{multline}
\prod_{j=0}^{2n}\left(\frac{\cos \left(\frac{1}{2} \left(m+\frac{2 j \pi }{1+2 n}+r\right)\right)+\sin
   \left(\frac{m-r}{2}\right)}{\cos \left(\frac{m+2 m n+2 j \pi +r+2 n r}{2+4 n}\right)-\sin
   \left(\frac{m-r}{2}\right)}\right)^{(-1)^j}\\
   =\left(-\frac{\sin \left(\frac{1}{4} (m (2+4 n)+\pi
   )\right)}{\tan \left(\frac{1}{4} (\pi +2 r+4 n r)\right) \sin \left(m \left(\frac{1}{2}+n\right)-\frac{\pi
   }{4}\right)}\right)^{(-1)^n}
\end{multline}
\end{example}
\begin{proof}
Use equation (\ref{fslf}) and form a second equation by replacing $m\to r$ take the difference of both these equations then set $k=-1,a=1$, take the exponential function of both sides and simplify using entry (3) of Section (64:12) in \cite{atlas}.
\end{proof}
\begin{example}
The product of the ratio of trigonometric and exponential of trigonometric functions.
\begin{multline}
\prod_{j=0}^{2n}\exp \left(-2 (-1)^j \sec \left(\frac{j \pi }{1+2 n}+x\right) \sec
   \left(\frac{j \pi }{1+2 n}+\frac{x}{b}\right)\right. \\ \left.
    \sin \left(\frac{(-1+b) x}{2
   b}\right) \sin \left(\frac{1}{2} \left(\frac{2 j \pi }{1+2
   n}+x+\frac{x}{b}\right)\right)\right)\\
    \left(\frac{\left(-1+\sin
   \left(\frac{j \pi }{1+2 n}+x\right)\right) \left(1+\sin \left(\frac{j \pi
   }{1+2 n}+\frac{x}{b}\right)\right)}{\left(1+\sin \left(\frac{j \pi }{1+2
   n}+x\right)\right) \left(-1+\sin \left(\frac{j \pi }{1+2
   n}+\frac{x}{b}\right)\right)}\right)^{\frac{(-1)^j}{2}}\\
   =\exp \left(\frac{2
   (-1)^n (1+2 n) \left(\cos (x+2 n x)-\cos \left(\frac{x+2 n
   x}{b}\right)\right)}{\cos \left(\frac{(-1+b) (x+2 n x)}{b}\right)+\cos
   \left(\frac{(1+b) (x+2 n x)}{b}\right)}\right)\\
    \left(\frac{(-1+\sin (x+2 n
   x)) \left(1+\sin \left(\frac{x+2 n x}{b}\right)\right)}{(1+\sin (x+2 n x))
   \left(-1+\sin \left(\frac{x+2 n
   x}{b}\right)\right)}\right)^{\frac{(-1)^n}{2}}
\end{multline}
\end{example}
\begin{proof}
Use equation (\ref{fslf}) and set $k=1,a=e,m=x$ and simplify using the method in section (8.1) in \cite{reyn_ejpam}.
\end{proof}
\begin{example}
Finite product of the ratio of the sine function. 
\begin{multline}
\prod_{j=0}^{2n}\left(\frac{2}{\sin \left(\frac{\pi  j}{2
   n+1}+\frac{x}{4}\right)+1}-1\right)^{(-1)^j} \left(\frac{\left(\sin
   \left(\frac{\pi  j}{2 n+1}+\frac{x}{2}\right)+1\right)^3 \left(\sin
   \left(\frac{\pi  j}{2 n+1}+x\right)-1\right)}{\left(\sin \left(\frac{\pi 
   j}{2 n+1}+\frac{x}{2}\right)-1\right)^3 \left(\sin \left(\frac{\pi  j}{2
   n+1}+x\right)+1\right)}\right)^{\frac{(-1)^j}{2}}\\
   =\left(\frac{2}{\sin
   \left(\frac{1}{4} (2 n x+x)\right)+1}-1\right)^{(-1)^n}
   \left(\frac{\left(\sin \left(\left(n+\frac{1}{2}\right) x\right)+1\right)^3
   (\sin (2 n x+x)-1)}{\left(\sin \left(\left(n+\frac{1}{2}\right)
   x\right)-1\right)^3 (\sin (2 n x+x)+1)}\right)^{\frac{(-1)^n}{2}}
\end{multline}
\end{example}
\begin{proof}
Use equation (\ref{fslf}) and set $k=1,a=1,m=x$ and simplify using the method in section (8.1) in \cite{reyn_ejpam}.
\end{proof}
\begin{example}
The finite product of the ratio of trigonometric and exponential of trigonometric functions.
\begin{multline}
\prod_{j=0}^{2n}\left(\frac{\left(\sin \left(\frac{\pi  j}{2 n+1}-x\right)-1\right)
   \left(\sin \left(\frac{\pi  j}{2
   n+1}-\frac{x}{2}\right)+1\right)}{\left(\sin \left(\frac{\pi  j}{2
   n+1}-x\right)+1\right) \left(\sin \left(\frac{\pi  j}{2
   n+1}-\frac{x}{2}\right)-1\right)}\right)^{\frac{1}{2} i \pi  (-1)^j}\\
   \left(\frac{\left(\sin \left(\frac{\pi  j}{2
   n+1}+\frac{x}{2}\right)+1\right) \left(\sin \left(\frac{\pi  j}{2
   n+1}+x\right)-1\right)}{\left(\sin \left(\frac{\pi  j}{2
   n+1}+\frac{x}{2}\right)-1\right) \left(\sin \left(\frac{\pi  j}{2
   n+1}+x\right)+1\right)}\right)^{\frac{1}{2} i \pi  (-1)^j}\\
    \exp \left((-1)^j
   \left(-\sec \left(\frac{\pi  j}{2 n+1}-x\right)\right.\right. \\ \left.\left.
   +\sec \left(\frac{\pi  j}{2
   n+1}-\frac{x}{2}\right)+\sec \left(\frac{\pi  j}{2
   n+1}+\frac{x}{2}\right)-\sec \left(\frac{\pi  j}{2
   n+1}+x\right)\right)\right)\\
   =\left(\frac{\left(\sin
   \left(\left(n+\frac{1}{2}\right) x\right)+1\right) (\sin (2 n
   x+x)-1)}{\left(\sin \left(\left(n+\frac{1}{2}\right) x\right)-1\right) (\sin
   (2 n x+x)+1)}\right)^{\frac{1}{2} i \pi  (-1)^n}\\
    \left(\frac{\left(\sin
   \left(\left(n+\frac{1}{2}\right) x\right)-1\right) (\sin (2 n
   x+x)+1)}{\left(\sin \left(\left(n+\frac{1}{2}\right) x\right)+1\right) (\sin
   (2 n x+x)-1)}\right)^{\frac{1}{2} i \pi  (-1)^n}\\
    \exp \left(2 (-1)^n (2 n+1)
   \left(\cos \left(\left(n+\frac{1}{2}\right) x\right)-1\right) \right. \\ \left.
   \left(\sec
   \left(\left(n+\frac{1}{2}\right) x\right)+2\right) \sec (2 n
   x+x)\right)
\end{multline}
\end{example}
\begin{proof}
Use equation (\ref{fslf}) and set $k=1,a=-1,m=x$ and simplify using the method in section (8.1) in \cite{reyn_ejpam}.
\end{proof}
\begin{example}
Finite product involving the ratio of the Gamma function.
\begin{multline}\label{eq:exp_lerch}
\prod_{j=0}^{2n}\exp \left(\frac{(-1)^{j-n} e^{\frac{i \pi  j}{2 n+1}}
   \Phi'\left(-e^{\frac{2 i \pi  j}{2
   n+1}},0,x\right)}{2 n+1}\right)=\frac{1}{\sqrt{2 (1+2 n)}}\frac{\Gamma \left(\frac{n+x}{2+4
   n}\right)}{ \Gamma \left(\frac{1+3 n+x}{2+4
   n}\right)}
\end{multline}
\end{example}
\begin{proof}
Use equation (\ref{fslf}) and set $m=0$ then take the first partial derivative with respect to $k$ and set $k=0,a=e^{ai}$ and simplify using equation (25.11.18) in \cite{dlmf}. Then take the exponential function of both sides and simplify.
\end{proof}
\begin{example}
The first partial derivative of the Hurwitz-Lerch zeta function  in terms of Euler's constant $\gamma$ and $\pi$.
\begin{multline}
\sum_{j=0}^{2n}(-1)^j e^{\frac{i \pi  j}{2 n+1}}
   \Phi'\left(-e^{\frac{2 i \pi  j}{2
   n+1}},1,\frac{1}{2}\right)\\
   =\frac{1}{4} \pi  (-1)^n \left(-2 \log (i (2
   n+1))+i \pi +2 \gamma +\log \left(\frac{64 \pi ^6}{\Gamma
   \left(\frac{1}{4}\right)^8}\right)\right)
\end{multline}
\end{example}
\begin{proof}
Use equation (\ref{fslf}) and take the first partial derivative with respect to $k$ and set $m=0,k=-1,a=1$ and simplify.
\end{proof}
\begin{example}
The first partial derivative of the Hurwtiz-Lerch zeta function  in terms of Catalan's constant $K$ and $\pi$.
\begin{equation}
\sum_{j=0}^{2n}(-1)^j e^{\frac{i \pi  j}{2 n+1}}
   \Phi'\left(-e^{\frac{2 i \pi  j}{2
   n+1}},-1,\frac{1}{2}\right)=\frac{K }{\pi }(-1)^n (2 n+1)^2
\end{equation}
\end{example}
\begin{proof}
Use equation (\ref{fslf}) and take the first partial derivative with respect to $k$ and set $m=0,k=1,a=1$ and simplify using equation (10) in \cite{wolfram} and equation (2.2.1.2.7) in \cite{lewin}
\end{proof}
\begin{example}
The first partial derivative of the Hurwtiz-Lerch zeta function  in terms of  the Glaisher-Kinkelin constant $A$.
\begin{multline}
\sum_{j=0}^{2n}(-1)^j e^{\frac{i \pi  j}{2 n+1}} \Phi'\left(-e^{\frac{2 i \pi  j}{2
   n+1}},-1,n+1\right)=\frac{1}{8} (-1)^n (2 n+1)^2 \log \left(\frac{A^{24}}{4\ 2^{2/3} e^2 (2
   n+1)^2}\right)
\end{multline}
\end{example}
\begin{proof}
Use equation (\ref{fslf}) and take the first partial derivative with respect to $k$ and set $m=0,k=1,a=1$ and simplify using equation (8) in \cite{wolfram}
\end{proof}
\begin{example}
The first partial derivative of the Hurwtiz-Lerch zeta function  in terms of zeta(3) is Ap\'{e}ry's constant $\zeta(3)$.
\begin{equation}
\sum_{j=0}^{2n}(-1)^j e^{\frac{i \pi  j}{2 n+1}} \Phi'\left(-e^{\frac{2 i \pi  j}{2
   n+1}},-2,n+1\right)=\frac{7 (-1)^n (2 n+1)^3 }{4 \pi ^2}\zeta (3)
\end{equation}
\end{example}
\begin{proof}
Use equation (\ref{fslf}) and take the first partial derivative with respect to $k$ and set $m=0,k=1,a=1$ and simplify using equation (9) in \cite{wolfram}
\end{proof}
\begin{example}
Special case of the Hurwitz-Lerch zeta function in terms of the Polylogarithm function.
\begin{multline}\label{polylog}
\sum_{j=0}^{2n}(-1)^j e^{\frac{i \pi  j}{2 n+1}} \Phi \left(-e^{2 i \left(\frac{\pi 
   j}{2 n+1}+m\right)},k,n+1\right)\\
   =-(-1)^n (2 n+1)^{1-k} e^{-2 i m (n+1)}
   \text{Li}_k\left(-e^{2 i m (2 n+1)}\right)
\end{multline}
\end{example}
\begin{proof}
Use equation (\ref{fslf}) and set $k\to -k,a= e^{i (2 n+1)}$ and simplify using equation (25.14.3) in \cite{dlmf}.
\end{proof}
\begin{example}
The first partial derivative of the Polylogarithm function in terms of $\pi$.
\begin{equation}
Li_{0}'\left(\sqrt[3]{-1}\right)+Li_{0}'\left(-(-1)^{2/3}\right)=\log
   \left(\frac{6 \sqrt{3}}{\pi }\right)
\end{equation}
\end{example}
\begin{proof}
Use equation (\ref{polylog}) and set $n=1,m=0$, then take the first partial derivative with respect to $k$ and set $k=0$ and simplify using equations (25.11.2), (25.14.2) and (25.14.3) in \cite{dlmf}. 
\end{proof}
\begin{example}
The first partial derivative of the Polylogarithm function in terms of Glaisher's constant $A$. 
\begin{equation}
Li_{-1}'\left(\sqrt[3]{-1}\right)+Li_{-1}'\left(-(-1)^{2/3}\right)=\log \left(\frac{36\ 2^{2/3}
   \sqrt[4]{3} e^2}{A^{24}}\right)
\end{equation}
\end{example}
\begin{proof}
Use equation (\ref{polylog}) and set $n=1,m=0$, then take the first partial derivative with respect to $k$ and set $k=1$ and simplify using equations (25.11.2), (25.14.2) and (25.14.3) in \cite{dlmf} and equation (8) in \cite{wolfram}.
\end{proof}
\begin{example}
The first partial derivative of the Polylogarithm function in terms of $\log(2)\log(3)$. 
\begin{equation}
Li_{1}'\left(\sqrt[3]{-1}\right)+Li_{1}'\left(1,-(-1)^{2/3}\right)=\log (2) \log (3)
\end{equation}
\end{example}
\begin{proof}
Use equation (\ref{polylog}) and set $n=1,m=0$, then take the first partial derivative with respect to $k$ and take the limit as $k\to -1$ and simplify using equations (25.11.2), (25.14.2) and (25.14.3) in \cite{dlmf} and equation (1) in \cite{wolfram} and the book by Lewin \cite{lewin}.
\end{proof}
\begin{example}
The first partial derivative of the Hurwitz-Lerch zeta function in terms of the log-gamma function. 
\begin{multline}
(-1)^{2/3}
   \Phi'\left(\sqrt[3]{-1},0,a\right)-\sqrt[3]{-1}
   \Phi'\left(-(-1)^{2/3},0,a\right)=\log
   \left(\frac{12 \sqrt{3} \Gamma \left(\frac{a+1}{2}\right) \Gamma
   \left(\frac{a+4}{6}\right)^3}{\Gamma \left(\frac{a}{2}\right) \Gamma
   \left(\frac{a+1}{6}\right)^3}\right)
\end{multline}
\end{example}
\begin{proof}
Use equation (\ref{fslf}) and set $m= -\frac{\pi }{2 n+1},n=1,a=e^{ai}$ and simplify using equation (25.11.18) in \cite{dlmf}. Then take the first partial derivative with respect to $k$ and set $k=0$ and simplify.
\end{proof}
\begin{example}
The Hurwitz-Lerch zeta function in terms of the digamma function.
\begin{multline}
\left(1-i \sqrt{3}\right) \Phi \left(\sqrt[3]{-1},1,a\right)+\left(1+i \sqrt{3}\right) \Phi \left(-(-1)^{2/3},1,a\right)\\
=-\psi ^{(0)}\left(\frac{a}{2}\right)-\psi
   ^{(0)}\left(\frac{a+1}{6}\right)+\psi ^{(0)}\left(\frac{a+1}{2}\right)+\psi ^{(0)}\left(\frac{a+4}{6}\right)
\end{multline}
\end{example}
\begin{proof}
Use equation (\ref{fslf}) and set $k=-1,m= -\frac{\pi }{2 n+1},n=1,a=e^{ai}$ and simplify using equation (25.11.18) in \cite{dlmf}.
\end{proof}
\begin{example}
The finite product of the ratio of trigonometric and exponential of trigonometric functions.
\begin{multline}
\prod_{j=0}^{2n}\left(\frac{\left(\sin \left(\frac{\pi  j}{2
   n+1}+\frac{x}{2}\right)+1\right) \left(\sin \left(\frac{\pi  j}{2
   n+1}+x\right)-1\right)}{\left(\sin \left(\frac{\pi  j}{2
   n+1}+\frac{x}{2}\right)-1\right) \left(\sin \left(\frac{\pi  j}{2
   n+1}+x\right)+1\right)}\right)^{(-1)^j (2 n+1)}\\
    \left(\cos \left(2 e^{i \pi 
   j} \left(\sec \left(\frac{\pi  j}{2 n+1}+\frac{x}{2}\right)-\sec
   \left(\frac{\pi  j}{2 n+1}+x\right)\right)\right)\right. \\ \left.
   -i \sin \left(2 e^{i \pi 
   j} \left(\sec \left(\frac{\pi  j}{2 n+1}+\frac{x}{2}\right)-\sec
   \left(\frac{\pi  j}{2 n+1}+x\right)\right)\right)\right)\\
   =2^{-2 (-1)^n (2
   n+1)} \left(i \sin \left(4 e^{i \pi  n} (2 n+1) \sin ^2\left(\frac{1}{4} (2
   n x+x)\right) \left(\sec \left(\left(n+\frac{1}{2}\right) x\right)+2\right)\right.\right. \\ \left.\left.
   \sec (2 n x+x)\right)+\cos \left(4 e^{i \pi  n} (2 n+1) \sin
   ^2\left(\frac{1}{4} (2 n x+x)\right) \right.\right. \\ \left.\left.
   \left(\sec
   \left(\left(n+\frac{1}{2}\right) x\right)+2\right) \sec (2 n
   x+x)\right)\right) \left(\csc \left(\frac{1}{4} (4 n x+2 x+\pi )\right)\right. \\ \left.
    \sec
   \left(\frac{1}{4} (2 n x+x+\pi )\right) \left(\cos \left(\frac{3}{4} (2 n
   x+x)\right)-\sin \left(\frac{1}{4} (2 n x+x)\right)\right)\right)^{2 (-1)^n
   (2 n+1)}
\end{multline}
\end{example}
\begin{proof}
Use equation (\ref{fslf}) and set $k=1,a= e^{i (2 n+1)},m=x$ and simplify using the method in section (8.1) in \cite{reyn_ejpam}.
\end{proof}
\begin{example}
The finite sum of the ratio of trigonometric functions.
\begin{multline}
\sum_{j=0}^{2n}\frac{e^{i j \left(1+\frac{1}{1+2 n}\right) \pi } }{\left(\cos (2 m)+\cos \left(\frac{2 j \pi }{1+2
   n}\right)\right)^2}\left((1+n) \left(\cos
   (2 m)+\cos \left(\frac{2 j \pi }{1+2 n}\right)\right)-i \sin \left(\frac{2 j
   \pi }{1+2 n}\right)\right)\\
   =\frac{(-1)^n (1+2 n)^2 \sin (2 m (1+n))}{2 \left(\sin (2
   m) \cos ^2(m (1+2) n)\right)}
\end{multline}
\end{example}
\begin{proof}
Use equation (\ref{polylog}) and form a second equation by replacing $m$ by $-m$ and taking their difference. Next set $k=-1$ and simplify. 
\end{proof}
\begin{example}
Finite summation formulae with the first generalized Stieltjes constant at complex argument.
\begin{multline}
\sum_{j=0}^{2n}(-1)^{j \left(\frac{1}{2 n+1}+1\right)}
   \Phi'\left(-e^{\frac{2 i \pi  j}{2
   n+1}},1,a\right)\\
   =\frac{1}{2} (-1)^n \left(-\gamma _1\left(\frac{a+n}{4
   n+2}\right)+\gamma _1\left(\frac{a+3 n+1}{4 n+2}\right)\right. \\ \left.
   +\log (4 n+2)
   \left(\psi ^{(0)}\left(\frac{a+n}{4 n+2}\right)-\psi ^{(0)}\left(\frac{a+3
   n+1}{4 n+2}\right)\right)\right)
\end{multline}
\end{example}
\begin{proof}
Use equation (\ref{fslf}) and set $m=0$ and simplify using equation (25.11.18) in \cite{dlmf}. Next apply l'Hopital's rule as $k\to -1$ and simplify using equation (3) in \cite{blagouchine}.
\end{proof}
\begin{example}
Finite summation formulae with the first and second generalized Stieltjes constant at complex argument.
\begin{multline}
\sum_{j=0}^{2n}(-1)^{j \left(\frac{1}{2 n+1}+1\right)}
   \Phi''\left(-e^{\frac{2 i \pi  j}{2
   n+1}},1,a\right)\\
   =\frac{1}{2} (-1)^n \left(\gamma _2\left(\frac{a+n}{4
   n+2}\right)-\gamma _2\left(\frac{a+3 n+1}{4 n+2}\right)\right. \\ \left.
   +2 \log (4 n+2)
   \left(\gamma _1\left(\frac{a+n}{4 n+2}\right)-\gamma _1\left(\frac{a+3
   n+1}{4 n+2}\right)\right)\right. \\ \left.
   -\left(\left(\log (2 n+1) \log (8 n+4)+\log
   ^2(2)\right) \left(\psi ^{(0)}\left(\frac{a+n}{4 n+2}\right)-\psi
   ^{(0)}\left(\frac{a+3 n+1}{4 n+2}\right)\right)\right)\right)
\end{multline}
\end{example}
\begin{proof}
Use equation (\ref{fslf}) and set $m=0$ and simplify using equation (25.11.18) in \cite{dlmf}. Next take the second partial derivative with respect to $k$ and set $k=1$ and simplify using equation (3) in \cite{blagouchine}.
\end{proof}
\begin{example}
Finite product of the exponential of the Polylogarithm function in terms of the ratio of gamma functions.
\begin{equation}
\prod_{j=0}^{2n}\exp \left(\frac{(-1)^{j-n} e^{-\frac{i \pi  j}{2 n+1}} Li_{0}'\left(-e^{\frac{2 i \pi  j}{2
   n+1}}\right)}{2 n+1}\right)=\sqrt{4 n+2}\left(\frac{ \Gamma \left(\frac{3 n+2}{4 n+2}\right)}{\Gamma \left(\frac{n+1}{4
   n+2}\right)}\right)
\end{equation}
\end{example}
\begin{proof}
Use equation (\ref{fslf}) and set $a=e^{i},m=0$ and simplify. Next take the first partial derivative with respect to $k$ and set $k=0$ and simplify using equation (25.11.8) in \cite{dlmf}. Then take the exponential function of both sides and simplify.
\end{proof}
\begin{example}
Finite sum of the logarithm function in terms of the digamma function.
\begin{multline}
\sum_{j=0}^{2n}(-1)^j e^{-\frac{i \pi  j}{2 n+1}} \log \left(1+e^{\frac{2 i \pi  j}{2 n+1}}\right)\\
=-\frac{1}{2} (-1)^n
   \left(\psi ^{(0)}\left(\frac{n+1}{4 n+2}\right)-\psi ^{(0)}\left(\frac{3 n+2}{4 n+2}\right)\right)
\end{multline}
\end{example}
\begin{proof}
Use equation (\ref{fslf}) and set $a=e^{i},m=0$ and simplify. Next apply l'Hopital's rule as $k\to -1$ and simplify. 
\end{proof}
\begin{example}
Finite sum of the Polylogarithm function in terms of the Hurwitz zeta function.
\begin{multline}
\sum_{j=0}^{2n}-1)^j e^{-\frac{i \pi  j}{2 n+1}} \text{Li}_{-k}\left(-e^{\frac{2 i j \pi }{2 n+1}}\right)\\
=i
   \left(\frac{i}{2}\right)^{-k} (-1)^n (i (2 n+1))^{k+1} \left(\zeta \left(-k,\frac{n+1}{4 n+2}\right)-\zeta
   \left(-k,\frac{3 n+2}{4 n+2}\right)\right)
\end{multline}
\end{example}
\begin{proof}
Use equation (\ref{fslf}) and set $a=e^i,m=0$ and simplify using equation (25.14.2) in \cite{dlmf}.
\end{proof}
\begin{example}
Finite trigonometric Euler form.
\begin{multline}
\prod_{j=0}^{2n}\left(1-\frac{2}{1+\sec \left(\frac{j \pi }{1+2 n}\right) \sin (m)}\right)^{(-1)^j}=-\left(\frac{\sin
   \left(\frac{m}{2}+m n-\frac{\pi }{4}\right)}{\sin \left(\frac{m}{2}+m n+\frac{\pi }{4}\right)}\right)^{2
   (-1)^n}
\end{multline}
\end{example}
\begin{proof}
Use equation (\ref{fslf}) and form a second equation by replacing $m\to -m$ take the difference of both these equations then set $k=-1,a=1$, take the exponential function of both sides and simplify using entry (3) of Section (64:12) in \cite{atlas}. Similar forms are given the the work by Chamberland \cite{chamberland}.
\end{proof}
\begin{example}
Finite sum in terms of a definite integral. 
\begin{multline}
\prod_{j=0}^{2n}\exp \left(\frac{(-1)^{j-n} e^{\frac{i \pi  j}{2 n+1}}
   \Phi'\left(-e^{\frac{2 i \pi  j}{2
   n+1}},0,x\right)}{2 n+1}\right)=\frac{1}{\sqrt{2 \pi } \sqrt{2 n+1}}\int_{0}^{\infty}\frac{z^{\frac{n+x}{4 n+2}-1}}{(z+1)^{\frac{3 n+x+1}{4 n+2}}}dz
\end{multline}
\end{example}
\begin{proof}
Use equation (\ref{eq:exp_lerch}) and equation (2) in \cite{tricomi}, compare the right-hand sides and simplify to achieve the stated result.
\end{proof}
\begin{center}
\begin{table}
\resizebox{15cm}{!} {
\begin{tabular}{ | l | c | r }
\hline
  \multicolumn{2}{|c|}{A short table of formulae} \\
\hline
  Feature & Formula \\
  \hline
  Log-gamma & $\sqrt[3]{-1} \Phi'\left((-1)^{2/3},0,a\right)-(-1)^{2/3}
   \Phi'\left(-\sqrt[3]{-1},0,a\right)
   =\log \left(\frac{2 \pi  3^{\frac{1}{2}-a}
   \Gamma (a)}{\Gamma \left(\frac{a+1}{3}\right)^3}\right)$ \\
  Catalan's constant, $K$ & $\sum_{j=0}^{2n}(-1)^j e^{\frac{i \pi  j}{2 n+1}}
   \Phi'\left(-e^{\frac{2 i \pi  j}{2
   n+1}},-1,\frac{1}{2}\right)=\frac{K }{\pi }(-1)^n (2 n+1)^2$ \\
  Glaisher's constant, $A$ & $\sum_{j=0}^{2n}(-1)^j e^{\frac{i \pi  j}{2 n+1}} \Phi'\left(-e^{\frac{2 i \pi  j}{2
   n+1}},-1,n+1\right)=\frac{1}{8} (-1)^n (2 n+1)^2 \log \left(\frac{A^{24}}{4\ 2^{2/3} e^2 (2
   n+1)^2}\right)$ \\
  Ap\'{e}ry's constant, $\zeta(3)$ & $\sum_{j=0}^{2n}(-1)^j e^{\frac{i \pi  j}{2 n+1}} \Phi'\left(-e^{\frac{2 i \pi  j}{2
   n+1}},-2,n+1\right)=\frac{7 (-1)^n (2 n+1)^3 }{4 \pi ^2}\zeta (3)$ \\
  Logarithm & $Li_{0}'\left(\sqrt[3]{-1}\right)+Li_{0}'\left(-(-1)^{2/3}\right)=\log
   \left(\frac{6 \sqrt{3}}{\pi }\right)$ \\
  Glaisher' constant, $A$ & $Li_{-1}'\left(\sqrt[3]{-1}\right)+Li_{-1}'\left(-(-1)^{2/3}\right)=\log \left(\frac{36\ 2^{2/3}
   \sqrt[4]{3} e^2}{A^{24}}\right)$ \\
  Product Logarithm & $Li_{1}'\left(\sqrt[3]{-1}\right)+Li_{1}'\left(1,-(-1)^{2/3}\right)=\log (2) \log (3)$ \\
  Log-gamma & $(-1)^{2/3}
   \Phi'\left(\sqrt[3]{-1},0,a\right)-\sqrt[3]{-1}
   \Phi'\left(-(-1)^{2/3},0,a\right)=\log
   \left(\frac{12 \sqrt{3} \Gamma \left(\frac{a+1}{2}\right) \Gamma
   \left(\frac{a+4}{6}\right)^3}{\Gamma \left(\frac{a}{2}\right) \Gamma
   \left(\frac{a+1}{6}\right)^3}\right)$ \\
  Gamma function & $\prod_{j=0}^{2n}\exp \left(\frac{(-1)^{j-n} e^{-\frac{i \pi  j}{2 n+1}} Li_{0}'\left(-e^{\frac{2 i \pi  j}{2
   n+1}}\right)}{2 n+1}\right)=\sqrt{4 n+2}\left(\frac{ \Gamma \left(\frac{3 n+2}{4 n+2}\right)}{\Gamma \left(\frac{n+1}{4
   n+2}\right)}\right)$ \\
   Euler product form & $\prod_{j=0}^{2n}\left(1-\frac{2}{1+\sec \left(\frac{j \pi }{1+2 n}\right) \sin (m)}\right)^{(-1)^j}=-\left(\frac{\sin
   \left(\frac{m}{2}+m n-\frac{\pi }{4}\right)}{\sin \left(\frac{m}{2}+m n+\frac{\pi }{4}\right)}\right)^{2
   (-1)^n}$ \\
    Definite integral & $\prod_{j=0}^{2n}\exp \left(\frac{(-1)^{j-n} e^{\frac{i \pi  j}{2 n+1}}
   \Phi'\left(-e^{\frac{2 i \pi  j}{2
   n+1}},0,x\right)}{2 n+1}\right)=\frac{1}{\sqrt{2 \pi } \sqrt{2 n+1}}\int_{0}^{\infty}\frac{z^{\frac{n+x}{4 n+2}-1}}{(z+1)^{\frac{3 n+x+1}{4 n+2}}}dz$ \\
    \hline
\end{tabular}}
\caption{Table of functional equations, sum and product formulae.}
    \label{tab:SumProdTables} 
\end{table}
\end{center}
%
%
%
\section{Conclusion}
This work offers a method that allows for the derivation of finite and infinite sums and product identities in terms of the the Hurwitz-Lerch zeta function and other Special functions. A summary table of some interesting formulae are summarized in Table (\ref{tab:SumProdTables}). For parameters in the derivations comprising real, imaginary, and complex values, the results were numerically confirmed using Wolfram's Mathematica software.
\end{document}